\documentclass[11pt, twoside]{article}
\usepackage{amsmath,amsthm,amssymb}
\usepackage{times}
\usepackage{enumerate}
\usepackage{latexsym}
\usepackage{amssymb}
\usepackage{amsmath}
\usepackage{fancybox}
\usepackage{epstopdf}
\usepackage{wasysym}
\usepackage{epsfig}
\usepackage{float}
\usepackage{amsfonts}
\usepackage{srcltx}
\usepackage{esint}
\usepackage[dvipsnames]{xcolor}
\usepackage{enumerate}

\numberwithin{equation}{section}

\def\titlerunning#1{\gdef\titrun{#1}}

\makeatletter

\def\author#1{\gdef\autrun{\def\and{\unskip, }#1}\gdef\@author{#1}}

\def\address#1{{\def\and{\\\hspace*{18pt}}\renewcommand{\thefootnote}{}%
\footnote {#1}}%
\markboth{\autrun}{\titrun}}

\makeatother

\def\email#1{e-mail: #1}

\def\keywords#1{\par\medskip
\noindent\textbf{Keywords.} #1}

\def\subjclass#1{\par\medskip
\noindent\textbf{Mathematics Subject Classification (2010).} #1}

\newcommand{\ut}{{\underline t}}
\newcommand{\ux}{{\underline x}}
\newcommand{\uQ}{{\underline Q}}

\newcommand{\und}{\underline}

\def\a{{\bf a}}

\def\rn{\mathbb{R}^{N}}
  \def\v{{\bf v}}
 \def\w{{\bf w}} \def\z{{\bf z}} 
 
\def\div{{\rm div}}
\DeclareMathOperator{\dive}{div}
\DeclareMathOperator{\supp}{supp}

\DeclareMathOperator*{\sign}{sign}
\DeclareMathOperator*{\esssup}{ess\,sup}
\DeclareMathOperator*{\essinf}{ess\,inf}

\renewcommand{\epsilon}{{\varepsilon}}

\newcommand{\R}{{\mathbb R}}

\renewcommand{\d}{\,{\mathrm d}}
 \newcommand{\eps}{{\varepsilon}}
 \def\1{\raisebox{2pt}{\rm{$\chi$}}}
\def\dys{\displaystyle}

\newcommand{\ignore}[1]{}

\theoremstyle{definition}

\newtheorem{theorem}{Theorem}[section]
\newtheorem{corollary}[theorem]{Corollary}
\newtheorem{lemma}[theorem]{Lemma}
\newtheorem{proposition}[theorem]{Proposition}
\newtheorem{definition}[theorem]{Definition}
\newtheorem{remark}[theorem]{Remark}

\newtheorem{assumption}[theorem]{Assumption}

\begin{document}

\titlerunning{Optimal waiting time bounds for some flux-saturated diffusion equations}

\title{Optimal waiting time bounds for some flux-saturated diffusion equations}

\author{Lorenzo Giacomelli \and Salvador Moll \and Francesco Petitta}

\date{}

\maketitle

\address{L. Giacomelli: SBAI Department,
 Sapienza University of Rome, Via Scarpa 16, 00161 Roma, Italy; \email{lorenzo.giacomelli@sbai.uniroma1.it}
\and
S. Moll: Departament d'An\`{a}lisi Matem\`atica,
Universitat de Val\`encia,  Spain; \email{j.salvador.moll@uv.es}
\and
F. Petitta: SBAI Department,
 Sapienza University of Rome, Via Scarpa 16, 00161 Roma, Italy; \email{francesco.petitta@sbai.uniroma1.it}
}

\begin{abstract}
We consider the Cauchy problem for two prototypes of flux-saturated diffusion equations. In arbitrary space dimension, we give an optimal condition on the growth of the initial datum which discriminates between occurrence or nonoccurrence of a waiting time phenomenon. We also prove optimal upper bounds on the waiting time. Our argument is based on the introduction of suitable families of subsolutions and on a comparison result for a general class of flux-saturated diffusion equations.



\end{abstract}

\section{Introduction}
\setcounter{equation}{0}

\subsection{Flux-saturated diffusion equations}

Flux-saturated diffusion equations are a class of second order parabolic equations of the form
\begin{equation}\label{general}
u_t=\dive \a(u,\nabla u),
\end{equation}
which are characterized by a hyperbolic scaling for large values of the modulus of the gradient, in the sense that
\begin{equation}\label{def-varphi}
\lim_{t\to +\infty} \a(z,t\v)\cdot\v = h^0(z,\v)
\quad \mbox{for all $z\ge 0$ and all $\v\in\rn$,}
\end{equation}
where $h^0:[0,+\infty)\times \rn\mapsto [0,+\infty)$ is positively $1$-homogeneous and convex in $\v$ (accounting for possible anisotropy effects), locally Lipschitz in $z$, and such that $h^0(\cdot,0)=0$ and $h^0(\cdot,\v)>0$ for all $\v\ne 0$. Here we restrict our attention to the case in which the saturated flux $h^0$ in \eqref{def-varphi} has the form
\begin{equation}\label{Phi-Str}
h^0(z,\v)=\varphi(z)\psi_0(\v)
\end{equation}
(cf. Remark \ref{somewhere}) and we are interested in the degenerate case, i.e., the case in which $\varphi(0)=0$.

\smallskip

To our knowledge, flux-saturated equations were first introduced in \cite{DuM} in the description of inertial confinement fusion, in which case $u$ represents the temperature. However, they find application whenever a saturation mechanism at high gradients, imposing a-priori bounds on speed or flux, is modeling-wise relevant for the phenomenon to be described (see for instance \cite{r1,r2,ckr, acms1,BW}). In addition, they emerge from a generalization of optimal transportation theory which accounts for relativistic-type cost functions (see \cite{br}). After pioneering contributions (\cite{bdp,bl2,dp}), the mathematical interest in flux-saturated equations is now steadily growing, leading to a well posedness theory for equations of the form \eqref{general}-\eqref{Phi-Str}. The theory is based on a suitable concept of entropy solution:
%
%
we refer to \S \ref{S-sol} for the precise definition and to \cite{C-JDE,C-PM,c15,cccss2,cccss} for recent overviews on modeling and analytical aspects.

\smallskip

Our focus is on two model equations which are known to approximate the porous medium equation (\cite{Caselles-Annali}): the {\it relativistic porous medium equation},
\begin{equation}\label{m}
u_t= \nu \dive \left(\frac{u^m \nabla u}{\sqrt{u^2+\nu^2 c^{-2}|\nabla u|^2}}\right), \quad m\in (1,+\infty),
\end{equation}
which generalizes the so-called relativistic heat equation ($m=1$), and the {\it speed-limited porous medium equation},
\begin{equation}\label{M}
u_t=\nu \dive \left(\frac{u \nabla u^{M-1}}{\sqrt{1+\nu^2 c^{-2}|\nabla u^{M-1}|^2}}\right),\quad M\in (1,+\infty)\,,
\end{equation}
where $\nu>0$ is a kinematic viscosity constant and $c>0$ represents a characteristic limiting speed. The former was proposed in \cite[Eq. (16)]{r2} with $m=3/2$ and in \cite[Eq. (34)]{br} with $m=1$, whereas the latter was proposed in \cite[Eq. (19)]{r2} (see also \cite{C-PM}). Up to the scaling $\hat t=\frac{c^2}{\nu}t$, $\hat x= \frac{c}{\nu} x$, we will hereafter assume without losing generality that
\begin{equation*}
\nu=c=1.
\end{equation*}

Equation \eqref{m} and \eqref{M} share common general features, such as finite speed of propagation of the support (\cite{G15}) and persistence of jump discontinuities (\cite{C-JDE}). However, they have remarkable differences, generated by the different scaling for large gradients: in one space dimension, a monotone increasing solution to \eqref{m}, resp. \eqref{M}, formally satisfies
\begin{equation}\label{heur1}
u_t \sim \left(u^m\right)_x \ \mbox{ for $u_x\gg 1$} , \quad\mbox{resp.}\quad
u_t\sim u_x \ \mbox{ for $(u^{M-1})_x\gg 1$.}
\end{equation}
This reflects into different qualitative behavior of solutions, highlighted also by numerical simulations as in  \cite{cmsv, ACMSV,CCM}. For instance, \eqref{heur1} suggests that \eqref{m} may yield to the {\em formation} of jump discontinuities if $m>1$, whereas \eqref{M} may not, and that the speed of propagation of the support is formally given by $u^{m-1}$ for \eqref{m} and by $1$ for \eqref{M}. For this reason, in the {former} case we conjecture that the formation of a discontinuity is not only sufficient (\cite{C-JDE}), but also necessary for the support to expand.

\subsection{Waiting-time phenomena: the main result}

The aforementioned difference manifests itself also in the {\it waiting time phenomenon}, a positive time before which the solutions' support does not expand around a point $x_0\in \R^N$. Starting from the porous medium equation (see \cite{V} for a review), this phenomenon is well known to occur for various classes of degenerate parabolic equations and systems, also of higher order (see e.g. \cite{DGG,GG,G,F} and references therein). Concerning \eqref{m} and \eqref{M}, after numerical and formal arguments in \cite{ACMSV,CCM}, rigorous sufficient conditions for a positive waiting time have been recently given in \cite{G15}: a positive constant $C$, depending only on $N$ and $m$ (resp. $M$), exists such that if
\begin{eqnarray}\label{wt1}
\esssup_{x\in\R^N} |x-x_0|^{-\frac{1}{m-1}} u_0(x)=L<+\infty && \quad\mbox{if $u$ solves \eqref{m}, or}
\\ \label{wt2}
\esssup_{x\in\R^N} |x-x_0|^{-\frac{2}{M-1}} u_0(x)=L<+\infty && \quad\mbox{if $u$ solves \eqref{M},}
\end{eqnarray}
then the entropy solution to the Cauchy problem for \eqref{m}, resp. \eqref{M}, is such that
\begin{equation}\label{wtr}
u(t, x_0)=0 \quad\mbox{for all}\quad t \le T_\ell:= \left\{\begin{array}{ll} {C L^{1-m}} & \mbox{if $u$ solves \eqref{m} }
\\
{C L^{1-M}}& \mbox{if $u$ solves \eqref{M}}
\end{array}
\right.
\end{equation}
(we refer to Section \ref{S-sol} for the definition of entropy solution). This result provides a {\em lower bound} $T_\ell$ on the waiting time. Based on \eqref{heur1}, in \cite{G15} it is also conjectured that these growth exponents are sharp. The main result of this paper confirms this fact.

\begin{theorem} \label{Tm-1}
Let $u_{0}\in L^\infty(\R^N)\cap L^1(\R^N)$ be nonnegative. Let $u$ be the entropy solution to the Cauchy problem for \eqref{m} (resp. \eqref{M}) with initial datum $u_{0}$ in the sense of Definition \ref{def-sol}. Let
$$
t_*=\sup\left\{t\ge 0: \ x_0\in \overline{\R^N\setminus \supp(u(\tau))} \quad \mbox{for all $\tau\in [0,t]$}\right\}.
$$
If $v_0 \in \mathbb S^{N-1}$ exists such that
\begin{equation}\label{crit-m}
\lim_{\rho\to 0^{+}} \essinf_{x\in B(x_0+\rho v_0,\rho)} u_{0}(x) |x-x_0|^{-\frac{1}{m-1}} \ge L\in (0,+\infty] \quad\mbox{if $u$ solves \eqref{m}},
\end{equation}
or
\begin{equation}\label{crit-M}
\lim_{\rho\to 0^{+}} \essinf_{x\in B(x_0+\rho v_0,\rho)} u_{0}(x) |x-x_0|^{-\frac{2}{M-1}} \ge L\in (0,+\infty] \quad\mbox{if $u$ solves \eqref{M}},
\end{equation}
then a positive constant $W$, depending on $m$ (resp. $M$) and $N$, exists such that
\begin{equation}\label{def-W}
t_*\le T_u:= \left\{\begin{array}{ll} {W L^{1-m}} & \mbox{if $u$ solves \eqref{m} }
\\
{W L^{1-M}}& \mbox{if $u$ solves \eqref{M}.}
\end{array}
\right.
\end{equation}
In particular, $t_*=0$ if $L=+\infty$.
\end{theorem}

The growth conditions \eqref{crit-m} and \eqref{crit-M} imply in particular that $\supp(u_0)$ satisfies an interior ball property at $x_0$,  i.e., $R>0$ exists such that $B(x_0+v_{0}R, R)\subset {\rm supp}(u_{0})$.

\smallskip

The results in Theorem \ref{Tm-1} are sharp. Indeed, comparing Theorem \ref{Tm-1} with \eqref{wt1}-\eqref{wt2} we see that the growth exponents in \eqref{crit-m}-\eqref{crit-M} are optimal. Note that the growth exponent $2/(M-1)$ coincides with that of the limiting porous medium equation, whereas $1/(m-1)$ does not. In addition, comparing Theorem \ref{Tm-1} with \eqref{wtr}, we see that the {\em upper bound} $T_u$ on the waiting time given in \eqref{def-W} is also optimal, in terms of scaling with respect to $L$.

\smallskip

The first main ingredient in our argument is a comparison result between solutions and subsolutions (see Theorem \ref{T-sub}). Based on Kruzhkov doubling method, a general approach for proving uniqueness of entropy solutions to degenerate flux-saturated equations has been introduced in \cite{NA-ACM,ACM} and later followed, or referred to, in quite a few subsequent papers (\cite{ACMM-ARMA, ACMM-JEE, ACMM-MA,ACM-JDE, C-DCDS, acms1, C-PM, CC, G15}). However, no comparison result is available when subsolutions are defined the way we need in our arguments (see Definition \ref{def-sub}). Therefore, in Section \ref{S-sol} we revisit the notion of (sub-)solution to the Cauchy problem for Eq. \eqref{general}, providing a general result on comparison with subsolutions for equations satisfying \eqref{def-varphi}-\eqref{Phi-Str} (see Assumption \ref{H} for details).

\smallskip

The second main ingredient in our argument is the introduction of suitable families of subsolutions, built such that optimal results may be obtained: their construction is outlined in the next subsection. With such subsolutions at hand, the strategy for Theorem \ref{Tm-1} becomes analogous to the one used for the porous medium equation (see \cite{V} and references therein). It is worked out in Sections 3 and 4, where the main result is proved: we argue by comparison, showing that subsolutions exist whose support is initially contained in $B(x_0+v_{0}R, R)$ and which expands up to $x_0$ within time $T_u$.

\smallskip

Besides comparison arguments, energy methods have also been developed in the analysis of waiting time phenomena in \cite{DGG,F,GG,G}. These methods are potentially capable of treating equations of general form (as opposed to explicit prototypes), leading to weaker, integral-type conditions on the initial datum. It would be interesting to explore the applicability of these methods to more general classes of flux-saturated diffusion equations of the form \eqref{general}.

\subsection{Classes of subsolutions}\label{ss-cs}

We now give a formal overview of the construction of subsolution in the case of \eqref{m}. As we mentioned, we expect that the support of solutions to \eqref{m} expands {\em only if} the solution has a jump discontinuity at the support's boundary. Therefore, it is natural to look for subsolutions which share the same property. Up to scaling and translation invariance, a prototype form is
$$
u(t,x)=\frac{1}{A(t)}f(r(t),|x|)\chi_{B(0,r(t))}(x), \quad f(r,y)=(1+(r^2-y^2)^\alpha), \quad\alpha>0,
$$
which is smooth in $B(0,r(t))$ with a moving front at $|x|=r(t)$. On this jump set, the upper and lower limits are given by $u^+(t,x)=1/A(t)$, resp. $u^-(t,x)=0$, and the inequality
\begin{equation}\label{sub-formal}
u_t\le \dive \left(\frac{u^m \nabla u}{\sqrt{u^2+|\nabla u|^2}}\right)
\end{equation}
formally translates into
\begin{equation}\label{lkj}
r'(t) \lim_{|x|\to r(t)^-} u(t,x) \le \frac{x}{r} \cdot \lim_{|x|\to r(t)^-} \frac{u^m \nabla u}{\sqrt{u^2+|\nabla u|^2}}.
\end{equation}
Provided that $\alpha<1$, we have $|\nabla u|\to +\infty$ as $x\to r(t)^-$. Therefore \eqref{lkj} reduces to $r'(t)\le A^{1-m}(t)$. In order to reach optimal results, consistently with the Rankine-Hugoniot condition we impose the equality:
\begin{equation}\label{crucial1-intro}
r'(t)= A^{1-m}(t) =\left.\frac{(u^+)^m-(u^-)^m}{u^+-u^-}\right|_{|x|=r(t)}.
\end{equation}
On the other hand, when $y\ll 1$, the degenerate parabolic structure dominates and \eqref{sub-formal} translates into
$$
u_t \lesssim  \dive \left(u^{m-1} \nabla u\right),
$$
that is,
\begin{eqnarray}
\dys \nonumber
-\frac{A'}{A^{2}}(1+r^{2\alpha})+\frac{2\alpha}{A} r^{2\alpha-1} r' \stackrel{\eqref{crucial1-intro}}=  -\frac{A'}{A^{2}}(1+r^{2\alpha})+ 2\alpha r^{2\alpha-1} A^{-m}
\\ \nonumber \\ \dys \label{lkj1}  \lesssim  - 2\alpha N A^{-m} r^{2\alpha-2} (1+r^{2\alpha})^{m-1}.
\end{eqnarray}
In order to enforce homogeneity of \eqref{lkj1} with respect to $A$, we choose
\begin{equation}\label{crucial2-intro}
A^{m-2}(t)A'(t)=\gamma
\end{equation}
for some constant $\gamma>0$. Combining \eqref{crucial1-intro} and \eqref{crucial2-intro} we obtain
$$
A(t) = \left((m-1)(1+\gamma t)\right)^{\frac{1}{m-1}}\,, \quad
r(t) = r_{0}+\frac{1}{\gamma (m-1)}\log(1+\gamma t)\,.
$$
As opposed to the porous medium equation, however, proving that such functions are indeed subsolutions is not obvious for two reasons: first, the crossover between the parabolic scaling (for $|y|\ll 1$) and the hyperbolic scaling (for $|r(t)-y|\ll 1$); second,  the nontrivial notion of subsolution (see Def. \ref{def-sub} below). On the other hand, the appropriate identification of $A$ and $r$ permits to obtain optimal results in terms of both growth exponent and waiting time bounds. Analogous arguments lead to a family of subsolutions for \eqref{M}, which up to scaling and translation invariances has the form
$$
u(t,x)= b^{\frac{1}{1-M}}\left(\ell-\frac{1}{1+wt}\right)^{\frac{1}{1-M}} \left(1-\frac{|x|^{2}}{(1+wt)^{2}}\right)^{\frac{1}{M-1}}_{+}
$$
for suitable $b>0$, $w>0$ and $\ell>0$ (see Section  \ref{S-M}).

\subsection{Notation}\label{ss-not}

For $a,b,\ell\in\R$ we let
\begin{eqnarray*}
&\mathcal T^+ = \{T_{a,b}^\ell: \ 0<a<b, \ \ell\le a\}, \quad\mbox{where}\quad T_{a,b}^\ell(r)=\max\{\min\{b,r\},a\}-\ell\ . &
\end{eqnarray*}
For a given function $T=T_{a,b}^\ell\in \mathcal T^+$, we denote with the superscript $0$ its translation of a height $\ell$: that is, we let $T^0:= T + \ell = T_{a,b}^0 \in \mathcal T^+$. For $f\in L^1_{loc}(\R)$ we let
$$
J_f(r):=\int_0^r f(s)\d s.
$$

We use standard notations and concepts for $BV$ functions as in \cite{AFP}; in particular, for $u\in BV(\R^N)$, $\nabla u\mathcal L^N$, resp.  $D^s u$, denote the the absolutely continuous, resp. singular, parts of $Du$ with respect to the Lebesgue measure $\mathcal L^N$, $J_u$ denotes its jump set and the approximate upper and lower limits of $u$ on $J_u$ are denoted by $u^+$ and $u^-$, respectively; i.e., $u^+(x)>u^-(x)$ for $x\in J_u$.
%
%

\section{Entropy (sub-)solutions}\setcounter{equation}{0}
\label{S-sol}

In this section we revisit the notion of entropy (sub-)solution to the Cauchy problem for \eqref{general}. We consider a function $\a$ satisfying the following properties:

\begin{assumption}
\label{H}{\rm
Let $Q=(0,\infty)\times \R^N$. The function $\a:\overline Q\to\R^N$ is such that:

\smallskip

\noindent $(i)$ ({\it Lagrangian}) there exists $f\in C(\overline Q)$ such that $\nabla_\v f=\a\in C(\overline Q)$, $f(z,\cdot)$ is convex, $f(z,0)=0$ for all $z\in [0,\infty)$, and
\begin{equation*} 
C_0(z)|\v| - D_0(z) \le f(z,\v)\le M_0(z)(1+|\v|) \quad\mbox{for all } (z,\v)\in Q
\end{equation*}
for nonnegative continuous functions $M_0,C_0\in C([0,\infty))$ and $D_0\in C((0,\infty))$, with $C_0(z)>0$ for $z>0$;

\smallskip

\noindent $(ii)$ ({\it flux}) $D_\v \a\in C(\overline Q)$; $\a(z,0)=\a(0,\v)=0$ and $h(z,\v):=\a(z,\v)\cdot \v=h(z,-\v)$ for all $(z,\v)\in \overline Q$; for any $R>0$ there exists $M_R>0$ such that
\begin{equation}\label{cvb}
|\a(z,\v)-\a(\hat z,\v)|\le M_R |z-\hat z| \quad\mbox{for all $z,\hat z\in [0,R]$ and all $\v\in\R^N$;}
\end{equation}

\noindent $(iii)$ ({\it recession functions}) the {\it recession functions} $f^0$ and $h^0$, defined by
$$
f^0(z,\v)=\lim_{t\to +\infty} \frac{1}{t} f(z,t\v), \quad h^0(z,\v)=\lim_{t\to +\infty} \frac{1}{t} h(z,t\v),
$$
exist in $\overline Q$; furthermore, a function $\varphi\in \mbox{Lip}_{loc}([0,\infty))$ with $\varphi(0)=0$ and $\varphi>0$ in $(0,\infty)$ and a  positive $1$-homogeneous convex function $\psi_{0}:\rn\mapsto \mathbb{R}$ {with $\psi_0(0)=0$ and $\psi_0(\v)>0$ for $\v\ne 0$} exist such that
\begin{equation}\label{def-varphi-bis}
f^0(z,\v)=h^0(z,\v)=\varphi(z)\psi_{0}(\v) \quad \mbox{for all $(z,\v)\in \overline Q$}
\end{equation}
(cf. \eqref{def-varphi} and \eqref{Phi-Str}), and
\begin{equation}\label{H3}
\mbox{$|\a(z,\w)\cdot \v|\le \varphi(z)\psi_{0}(\v)$  for all $(z,\v)\in Q$, $\w\in \R^N$.}
\end{equation}
}
\end{assumption}

The convexity of $f$ implies that
$$
(\a(z,\v)-\a(z,\hat \v))\cdot (\v-\hat \v)\ge 0 \quad\mbox{for all $\v,\hat \v\in\R^N$}
$$
which, combined with \eqref{cvb}, also yields
\begin{equation}\label{H2}
(\a(z,\v)-\a(\hat z,\hat \v))\cdot (\v-\hat \v)\ge - M_R|z-\hat z||\v-\hat \v|
\end{equation}
for all $z,\hat z\in [0,R]$ and all $\v,\hat \v\in\R^N$.

\smallskip

The concept of entropy solution to the Cauchy problem for \eqref{general} has been introduced in \cite{ACM} and later  extended in \cite{ACMM-ARMA, C-DCDS, C-PM}. At the core of this concept is an entropy inequality (cf. \eqref{main-ineq} below) which follows from formally testing \eqref{general} by $\phi S(u)T(u)$ with $S,T\in \mathcal T^+$ and $\phi$ smooth and nonnegative. In particular, when constructing a solution as limit of solutions to suitable approximating problems, one needs to argue by lower semi-continuity on terms of the form
$$
S(u)a(u,\nabla u)\cdot \nabla T(u) = S(T^0(u))h(T^0(u),\nabla T^0(u))
$$
(see the discussion in \cite[\S 2.2 and 3.2]{ACM}). This leads to the following entropy inequality:
\begin{eqnarray}
\nonumber
\lefteqn{\int_0^{+\infty} \langle h_S(u,DT(u)) +h_T(u,D S(u)),\phi \rangle \d t }
\\ \label{main-ineq}
&\le &  \int_0^{+\infty} \int_{\R^N} \left(J_{TS}(u)\phi_t - T(u)S(u) \a(u,\nabla u) \cdot \nabla \phi\right) \d x\d t \,,
\end{eqnarray}
where $h_S(u,DT(u))$ is the Radon measure defined by
\begin{eqnarray}\nonumber
\lefteqn{\langle h_S(u,DT(u)), \phi\rangle  := \int_{\R^N} \phi S(T^0(u)) h(T^0(u),\nabla T^0(u))\d x} \\ && + \int_{\R^N} \phi \psi_{0}\left( \frac{DT^{0}(u)}{|D T^{0}( u)|}\right) \d |D^s J_{S\varphi}(T^0(u))| \quad\mbox{for all $\phi\in C_c(\R^N)$}
 \label{lsc}
\end{eqnarray}
and $\varphi$,$\psi_{0}$ are defined through \eqref{def-varphi-bis}. This motivates the following definition:

\begin{definition}\label{def-sol}
Let $\a$ be such that Assumption \ref{H} holds and let $u_0\in L^\infty(\R^N)\cap L^1(\R^{N})$ be nonnegative. A nonnegative function $u\in C([0,+\infty); L^1(\R^N))\cap L^\infty((0,\infty)\times\R^N)$ is an entropy solution to the Cauchy problem for \eqref{general} with initial datum $u_0$ if $u(0)=u_0$ and:
\begin{itemize}
\item[$(i)$] $T^a_{a,b}(u)\in L^1_{loc}((0,+\infty); BV(\R^N))$ for all $0<a<b$;
\item[$(ii)$] $u_t=\dive (\a(u,\nabla u))$ in the sense of distributions;
\item[$(iii)$] inequality \eqref{main-ineq} holds for any $S,T\in \mathcal T^+$ and any nonnegative $\phi \in C^\infty_c((0,+\infty)\times \R^N)$.
\end{itemize}
\end{definition}

Definition \ref{def-sol} implies mass conservation:
\begin{proposition}\label{prop-mass-cons}
Any solution $u$ in the sense of Definition \ref{def-sol} is such that
\begin{equation}\label{eq-mass-cons}
\int_{\R^N} u(t,x)\d x =\int_{\R^N} u_0(x) \d x \quad {\rm for \ all \ }t\in [0,+\infty).
\end{equation}
\end{proposition}

\begin{proof}
Let $\eta_R\in {\mathcal{D}}(\R^N)$ be an increasing sequence of nonnegative functions such that $\eta_R=1$ on $B(0,R)$, $\eta_R=0$ in $\R^N\setminus B(0,R+1)$, and $|\nabla\eta_R|\le C$. Let $\psi_\eps(t):=\chi_{\{[t_1,t_2]\}}\ast \rho_\eps$, where $\rho_\eps$ is a standard mollifier and $[t_1,t_2]\subset (0,+\infty)$. We donote by $C$ a generic positive constant independent of $\eps$ and $R$. Testing $(ii)$ in Definition \ref{def-sol} with $\psi_\eps(t)\eta_R(x)$ (with $\eps$ sufficiently small) and integrating by parts we obtain
\begin{eqnarray*}
-\int_0^{+\infty}\int_{\R^N} u \eta_R \psi'_{\eps}\d x\d t = \int_0^{+\infty} \int_{\R^N} \psi_\eps\a(u,\nabla u)\cdot \nabla\eta_R \d x\d t.
\end{eqnarray*}
Since $u\in C([0,+\infty);L^1(\R^N))$, letting $\eps\to 0$ we obtain
\begin{eqnarray}\label{fgh}
\left.\int_{\R^N} u \eta_R \d x\right|_{t=t_1}^{t=t_2} = \int_{t_1}^{t_2}\int_{\R^N} \a(u,\nabla u)\cdot \nabla\eta_R \d x\d t.
\end{eqnarray}
Since $u\in L^\infty((0,\infty)\times\R^N)$, \eqref{cvb} implies that $|\a(u,\nabla u)|\le C u$. Therefore
\begin{eqnarray*}
\left|\int_{t_1}^{t_2}\int_{\R^N} \a(u,\nabla u)\cdot \nabla\eta_R \d x\d t\right| &  \le & C \int_{t_1}^{t_2}\int_{B(0,R+1)\setminus B(0,R)} u \d x\d t \to 0 \quad\mbox{as \ $R\to +\infty$.}
\end{eqnarray*}
Since $u\ge 0$, the two terms on the left-hand side of \eqref{fgh} pass to the limit as $R\to +\infty$ by monotone convergence. Therefore
$$
\int_{\R^N} u(t_2,x)\d x=\int_{\R^N} u(t_1,x)\d x.
$$
Again since $u\in C([0,+\infty);L^1(\R^N))$, passing to the limit as $t_1\to 0^+$ we obtain \eqref{eq-mass-cons}.
\end{proof}

\begin{remark}\label{here} Existence and uniqueness of entropy solutions to the Cauchy problem for equations \eqref{m} and \eqref{M} are contained in, or follow from, earlier results in \cite{ACM}, resp. \cite{C-PM}. We refer e.g. to \cite{G15} for details. In fact, \cite{ACM,C-PM} contain existence and uniqueness results for general classes of equations \eqref{general} satisfying Assumption \ref{H} together with slight additional hypotheses.
\end{remark}

To our purposes, we use the  notion of subsolution for equations of the form \eqref{general} suggested by Caselles in \cite[Section 3.3]{C-DCDS}.  Such notion  is analogous to the one of (entropy) solution, except that $(ii)$ in Definition \ref{def-sol} is not required.

\begin{definition}\label{def-sub}
Let $\tau>0$ and let $\a$ such that Assumption \ref{H} holds. A nonnegative function $u\in C([0,\tau); L^1(\R^N))\cap$ $L^\infty([0,\tau]\times\R^N)$ is an entropy subsolution to equation \eqref{general} in $(0,\tau)\times \R^N$ if:
\begin{itemize}
\item[$(i)$] $T_{a,b}^a(u)\in L^1_{loc}((0,\tau); BV(\R^N))$ for all $0<a<b$;
\item[$(ii)$] inequality \eqref{main-ineq} holds for any $S,T\in \mathcal T^+$ and any nonnegative $\phi \in C^\infty_c((0,\tau)\times \R^N)$.
\end{itemize}
\end{definition}

This notion of subsolution yields the following comparison result:

\begin{theorem}\label{T-sub}
Let $\tau>0$ and let $\a$ such that Assumption \ref{H} holds. Let $u$ be an entropy solution to the Cauchy problem for \eqref{general} with initial datum $u_0 \in L^\infty(\R^N)\cap L^1(\R^N)$ and $\underline u$ be an entropy subsolution to equation \eqref{general} in $(0,\tau)$. If $\und u(0)\le u_0$, then $\underline u(t)\le u(t)$ for all $t\in (0,\tau)$.
\end{theorem}

\begin{remark}\label{rem-appl}
Theorem \ref{T-sub} applies in particular to equations \eqref{m} and \eqref{M} with $\varphi(s)=s^{m}$, resp. $\varphi(s)=s$, and $\psi_{0}(\v)=|\v|$; we refer to Remark 1.3 in \cite{G15} for details.
\end{remark}

\begin{remark}\label{somewhere}
To our knowledge, known results on existence, uniqueness, and comparison for flux-saturated equations assume, as in Theorem \ref{T-sub}, some form of separation of variables for the limiting flux $h^0$ (mostly the one in \eqref{Phi-Str} and \eqref{def-varphi-bis}, but see also \cite{c15}). It would be interesting to see whether this assumption can be removed, treating a more general $h^0(z,\v)$ which is $1$-homogeneous and convex in $\v$ and locally Lipschitz in $z$. Though at the cost of further technical complications, we expect it to be possible. In any event, Assumption \ref{H} covers the model equations \eqref{m} and \eqref{M} (cf. Remark \ref{rem-appl}) and suffices to the scope of this work.
\end{remark}

As we mentioned in the Introduction, no comparison result is available when subsolutions are defined as in Definition \ref{def-sub}. Therefore, below we provide a complete and self-contained proof of Theorem \ref{T-sub}. The proof follows the approach introduced in \cite{NA-ACM,ACM}: however, it also clarifies and simplifies some of the arguments, such as the choice of testing functions (see the comment after \eqref{2UE6}) and the estimate of $I_2$ (see \eqref{def-I2}-\eqref{UEC}), easing the overall presentation.

\begin{proof}[Proof of Theorem \ref{T-sub}]
We will apply Kruzkhov's doubling argument, which is known to give rigorous basis to the formal multiplication of $(u-\underline u)_t$ by $\sign(u-\underline u)_+$. To this aim, in the entropy inequality \eqref{main-ineq} for $u$ we will replace $S$ by sequences $R_{\eps,l}(u)/\eps$ which approximate $\sign(u-l)_+$ as $\eps\to 0$. In addition, in order to deal with the lack of regularity of $u$ at $u=0$, we will replace $T(u)$ by sequences $T_{a,b}^a(u)/(b-a)$ approximating $\sign_+(u)$ as $a\to 0$ and $b\to 0$, in this order (in fact, this is the main motivation for the presence of two testing functions in \eqref{main-ineq}). Analogous choices will be made in the entropy inequality \eqref{main-ineq} for $\underline u$, with $S$ replaced by sequences $S_{\eps,l}(\underline u)/\eps$ approximating $\sign(\underline u-l)_+$ as $\eps\to 0$.

\smallskip

{\it (1). Doubling.} Let $b > a > 2\eps > 0$, $l\ge 0$, $T(r) =T_{a,b}^{a}(r)$, and $T_{a,b}(r)=T^0_{a,b}(r)$. We denote $\z=\a(u,\nabla u)$, $\underline \z=\a(\underline u, \nabla \underline u)$,
\begin{eqnarray}\label{def-R}
R_{\eps, l}(r) &:=& \left\{\begin{array}{ll} T_{l-\eps,l}(r) -(l-\eps) & \mbox{if $l>2\eps$},
\\
T_{\eps,2\eps}(r) -\eps & \mbox{if $l<2\eps$},
\end{array}\right.
\\ \label{def-S}
S_{\eps,l}(r) &:=& \left\{\begin{array}{ll} T_{l,l+\eps}(r) -l & \mbox{if $l>\eps$},
\\
T_{\eps,2\eps}(r) -\eps & \mbox{if $l<\eps$}.
\end{array}\right.
\end{eqnarray}
We choose two different pairs of variables $(t,x)\in Q_\tau$, $(\ut, \ux)\in \uQ_\tau:=Q_\tau$, and consider $u$, $\z$ and $\und u$, $\und \z$ as functions of $(t,x)$, resp. $(\ut,\ux)$. Let $0 \leq \phi \in {\mathcal D}((0, \tau))$, $\rho_m$ a sequence of mollifiers in $\R^N$, and $\tilde{\rho}_n$ a sequence of mollifiers in $\R$. Define
$$
\eta_{m,n}(t, x, \ut, \ux):= \rho_m(x - \ux) \tilde{\rho}_n(t - \ut) \phi
\bigg(\frac{t+\ut}{2} \bigg).
$$
For $(\ut,\ux)$ fixed, choosing $S=R_{\eps,\underline{u}}$ in \eqref{main-ineq} we obtain
\begin{eqnarray}\label{2UE5}
\nonumber
- \int_{Q_\tau} {J}_{T,R_{\eps,\underline{u}}}(u)(\eta_{m,n})_t
&+& \int_{Q_\tau} \eta_{m,n} \d \left(h_T(u,
D_x R_{\eps,\underline{u}}(u) ) + h_{R_{\eps,
\underline{u}}}(u,D_x T(u))\right)
\\ & + & \int_{Q_\tau} \z
\cdot \nabla_x \eta_{m,n} \ T(u) \ R_{\eps,
\underline{u}}(u) \leq 0.
\end{eqnarray}
Similarly, for $(t,x)$ fixed, choosing $S=S_{\eps,u}$ in \eqref{main-ineq} we obtain
\begin{eqnarray}\label{2UE6}
\nonumber - \int_{\uQ_\tau} {J}_{T,S_{\eps,u}}(
\underline{u}) (\eta_{m,n})_\ut
&+& \int_{\uQ_\tau} \eta_{m,n}\d
\left(h_T(\underline{u}, D_\ux S_{\eps,u}( \underline{u}) ) +
h_{S_{\eps,u}}(\underline{u},D_\ux T(\underline{u}))\right)
\\ &+& \displaystyle \int_{\uQ_\tau} \underline{\z}
\cdot \nabla_\ux  \eta_{m,n} \ T(\underline{u}) \ S_{\eps,
u}(\underline{u})\leq 0.
\end{eqnarray}
It seems that one can not directly choose $R_{\eps, \underline u} = T_{\underline u-\eps,\underline u } -(\underline u -\eps)$ as test function: indeed, $T_{\underline u-\eps,\underline u} \notin \mathcal T^+$  when $\underline u<\eps$ (analogous considerations hold for $S_{\eps,u}= T_{u,u+\eps} -u$ when $u=0$). This motivates the definitions in \eqref{def-R} and \eqref{def-S}. However, as shown in \eqref{recover1}-\eqref{recover2} below, such simpler form will be recovered after doubling variables and integrating by parts, due to the presence of the second truncating function $T$.

\smallskip

Integrating (\ref{2UE5}) in $\uQ_\tau$, (\ref{2UE6}) in $Q_\tau$, adding the two inequalities, taking into account that $\nabla_x \eta_{m,n} + \nabla_\ux  \eta_{m,n} = 0$, and noting that
$$
\int_{Q_\tau\times \uQ_\tau} \eta_{m,n} \d \left(h_{R_{\eps,\underline{u}}}(u,D_x T(u)) +
h_{S_{\eps,u}}(\underline{u},D_\ux T(\underline{u}))\right) \ge 0,
$$
we see that
\begin{equation}\label{2UE8}
I_1+I_2\le 0,
\end{equation}
where
\begin{eqnarray}\nonumber
I_1 &:=& -\int_{Q_\tau \times \uQ_\tau}
\left({J}_{T,R_{\eps,\underline{u}}}(u) (\eta_{m,n})_t +
{J}_{T,S_{\eps,u}}( \underline{u}) (\eta_{m,n})_\ut \right)
\\ \nonumber
I_2 &:= & \int_{Q_\tau\times \uQ_\tau}\eta_{m,n} \d  \left(h_T(u, D_x R_{\eps, \underline{u}}(u) + h_T(\underline{u}, D_\ux
S_{\eps,u}(\underline{u}))\right)
\\  \nonumber && - \int_{Q_\tau
\times \uQ_\tau} \underline{\z} \cdot
 \nabla_x \eta_{m,n}  T(\underline{u}) S_{\eps,u}(\underline{u})
- \int_{Q_\tau
\times \uQ_\tau} \z \cdot
\nabla_\ux  \eta_{m,n} T(u) R_{\eps, \underline{u}}(u).
\end{eqnarray}

{\it (2). The limit $\eps\to 0$.} We now divide \eqref{2UE8} by $\eps$ and let $\eps\to 0$. Concerning $I_1$, we note that
\begin{equation}\label{def-JE}
\frac{1}{\eps} {J}_{T,R_{\eps,l}}(r) = \int_0^r T(s) \frac{R_{\eps,l}(s)}{\eps} \d s \to J_T^l(r):=\int_0^r T(s) \sign (s-l)_+ \d s
\end{equation}
and, analogously, ${J}_{T,S_{\eps,l}}(r)\to {J}_T^l(r)$, as $\eps\to 0$.
Therefore, by dominated convergence,
\begin{equation}\label{ineqforj}
\frac{I_1}{\eps}\stackrel{\eps\to 0}\to -\int_{Q_\tau\times \uQ_\tau}\left((\eta_{m,n})_t J_T^{\underline u}(u) + (\eta_{m,n})_\ut J_T^u(\underline u)\right).
\end{equation}
Concerning $I_2$, we will argue that $\liminf_{\eps\to 0}{I_2/\eps}\ge 0$. After one integration by parts we obtain
\begin{eqnarray}\nonumber
I_2 &=& \int_{Q_\tau\times \uQ_\tau} \eta_{m,n} \d h_T(u, D_x R_{\eps, \underline{u}}(u)) + \int_{Q_\tau\times \uQ_\tau}
\eta_{m,n} T(\underline{u}) \ \underline{\z} \cdot \d D_x {S_{\eps,u}(\und u)}
\\
&+& \int_{Q_\tau\times \uQ_\tau} \eta_{m,n} \d h_T(\underline{u}, D_\ux  S_{\eps,u}(\underline{u}))  + \int_{Q_\tau\times \uQ_\tau}
\eta_{m,n}  T(u) \ \z \cdot \d D_\ux {R_{\eps,\und u}(u)}.\label{def-I2}
\end{eqnarray}
Due to the presence of $T$, the second and the fourth integrands in $I_{2}$ are nonzero only on $\{\und u>a\}$, resp. $\{u>a\}$. Moreover, for $r>a$, we have
$$
R_{\eps,l}(r)= \left\{\begin{array}{ll} T_{l-\eps,l}(r) -(l-\eps) & \mbox{if $l>a$}
\\
T_{l-\eps,l}(r) -(l-\eps)=\eps & \mbox{if $2\eps<l<a$}
\\
T_{\eps,2\eps}(r) -\eps =\eps & \mbox{if $l<2\eps$}
\end{array}\right\} = T_{l-\eps,l}(r)-(l-\eps) \quad\mbox{for $r>a$}.
$$
Analogously, $S_{\eps,l}(r)=T_{l,l+\eps}(r)-l$ for $r>a$. Therefore, in $I_2$ we may equivalently consider
\begin{eqnarray}\label{recover1}
R_{\eps,\und u}(u) &=& T_{\und u-\eps,\und u}(u)-(\und u-\eps) = T_{0,\eps}(u-\und u+\eps),
\\ \label{recover2} S_{\eps,u}(\und u) &=& T_{u,u+\eps}(\und u)-u= T_{0,\eps}(\und u-u).
\end{eqnarray}
The latter equalities in \eqref{recover1}-\eqref{recover2} show in particular that $R_{\eps,\und u}(u)+S_{\eps,u}(\und u)\equiv \eps$. Therefore,
$$
D_x R_{\eps,\und u}(u)= -D_x S_{\eps,u}(\und u) \quad\mbox{and}\quad D_\ux S_{\eps,u}(\und u) =-D_\ux R_{\eps, \und u} (u).
$$
Furthermore, letting
\begin{equation}\label{def-ue}
u_\eps := T_{\und u-\eps,\und u}(u), \quad  \und u_\eps := T_{u,u+\eps}(\und u),
\end{equation}
it follows from the former equalities in \eqref{recover1}-\eqref{recover2} that
$$
D_x R_{\eps,\und u}(u)= D_x u_\eps \quad\mbox{and}\quad D_\ux S_{\eps,u}(\und u) =D_\ux \und u_\eps.
$$
Altogether, $I_2$ may be rewritten as
\begin{eqnarray*}
I_2 &=& \int_{Q_\tau\times \uQ_\tau} \eta_{m,n} \d h_T(u, D_x u_\eps) - \int_{Q_\tau\times \uQ_\tau}
\eta_{m,n} T(\underline{u}) \ \underline{\z} \cdot \d D_x u_\eps
\\
&+& \int_{Q_\tau\times \uQ_\tau} \eta_{m,n} \d h_T(\underline{u}, D_\ux \underline{u}_\eps)  - \int_{Q_\tau\times \uQ_\tau}
\eta_{m,n}  T(u) \ \z \cdot \d D_\ux \und u_\eps.
\end{eqnarray*}
Let us write$I_2 = I_2(ac) + I_2(s)$, where $I_2(ac)$ and $I_2(s)$ contain the absolutely continuous, resp. singular, part of the measures involved in $I_2$. Let us first consider $I_2(ac)$. Letting
\begin{equation}\label{def-chie}
\chi_\eps:= \chi_{\{u<\underline u<u+\eps\}} = \chi_{\{\underline u-\eps <u <\und u\}},
\end{equation}
in view of \eqref{def-ue} and \eqref{recover1}-\eqref{recover2} we have $\nabla_x u_\eps=\chi_\eps \nabla_x u$ and  $\nabla_\ux \und u_\eps =\chi_\eps \nabla_\ux \und u$. Therefore, recalling \eqref{lsc}, $I_2(ac)$ may be rewritten as
\begin{eqnarray*}
I_2(ac) &=& \int_{Q_\tau
\times \uQ_\tau} \eta_{m,n} \chi_\eps (T(u)\z-T(\underline u)\underline \z)\cdot (\nabla_x u-\nabla_\ux \underline u)
\\ &=& \int_{Q_\tau
\times \uQ_\tau} \eta_{m,n} \chi_\eps (T(u)-T(\underline u)) \z\cdot (\nabla_x u-\nabla_\ux  \underline u)
\\ && + \int_{Q_\tau
\times \uQ_\tau} \eta_{m,n} \chi_\eps T(\underline u)(\z-\underline \z)\cdot (\nabla_x u-\nabla_\ux  \underline u)
\\ &=:& I_{2,1}(ac)+ I_{2,2}(ac).
\end{eqnarray*}
In view of \eqref{cvb} and since $u\in L^\infty(Q)$, we have that $\|\z\|_\infty \le M$. In addition, \eqref{def-chie} implies that
\begin{equation}\label{chi2}
\chi_\eps\cap\{\und u>a\}\subseteq\{u>a-\eps\}\ \  \text{and}\quad  \chi_\eps\cap\{ u>a \}\subseteq \{\und u>a\}.
\end{equation}
Therefore, since $T$ is $1$-Lipschitz,
\begin{eqnarray}\label{f1}\nonumber
\dys \frac{1}{\eps}|I_{2,1}(ac)|  &\leq  &\frac{{M}}{\eps}\|\eta_{m,n}\|_\infty \int_{Q_\tau
\times \uQ_\tau} \chi_\eps \chi_{\{u>a-\eps\}}\chi_{\{\underline u>{a}\}}
|u-\underline u| |\nabla_x u-\nabla_\ux  \underline u|
\\
\dys  & \stackrel{\eqref{def-chie}}\leq & M \|\eta_{m,n}\|_\infty \int_{Q_\tau
\times \uQ_\tau} \chi_\eps\chi_{\{u>a-\eps\}}\chi_{\{\underline u>{a}\}}
(|\nabla_x u|+| \nabla_\ux  \underline u|).
\end{eqnarray}
Similarly,
\begin{eqnarray}\label{f2}
\dys \nonumber
\frac{1}{\eps} I_{2,2}(ac) & \stackrel{\eqref{H2}}\ge & - \frac{M}{\eps} \int_{Q_\tau
\times \uQ_\tau} \eta_{m,n} \chi_\eps T(\underline u)|u-\underline u|\ |\nabla_x u-\nabla_\ux \und u|
\\ \dys & \stackrel{\eqref{chi2}}\ge &  - M \|\eta_{m,n}\|_\infty \int_{Q_\tau
\times \uQ_\tau} \chi_\eps \chi_{\{\underline u>a\}}\chi_{\{u>a-\eps\}} {\left(|\nabla_x u|+|\nabla_\ux \und u|\right)}.
\end{eqnarray}
We claim that
\begin{equation}\label{coarea}
\lim_{\eps\to 0} \int_{Q_\tau\times\uQ_\tau} \chi_{\{\und u>a\}} \d |D_xu_\eps| = \lim_{\eps\to 0} \int_{Q_\tau\times \uQ_\tau} \chi_{\{u>a-\eps\}} \d |D_\ux\und u_\eps| =0.
\end{equation}
Let us show the first one (the second is identical). By the coarea formula, we have for any $l>a$
\begin{eqnarray*}
\int_{\R^N} |D_x T_{l-\eps,l}(u)| &=& \int_{l-\eps}^l P(T_{l-\eps,l}(u)>\lambda)\d \lambda = \int_{l-\eps}^l P(T_{a/2,\ell}(u)>\lambda)\d \lambda \to 0
\end{eqnarray*}
as $\eps\to 0$, since $\lambda \mapsto P(T_{a/2,\ell}(u)>\lambda)$ is integrable in $\R$. Then \eqref{coarea} follows from dominated convergence, since
$$
(t,\ux,\ut)\mapsto \chi_{\{\und u>a\}} \int_{\R^N} \d |D_x u_\eps| \le \chi_{\{\und u>a\}} \int_{\R^N} \d |D_x T_{a/2,\ell}(u)| \in L^1((0,\infty)^2\times \R^N).
$$
Since $\chi_\eps|\nabla_x u|=|\nabla_x u_\eps|$ and $\chi_\eps|\nabla_\ux \und u|=|\nabla_\ux \und u_\eps|$, combining \eqref{f1}, \eqref{f2}, and \eqref{coarea} we conclude that
\begin{equation}
  \label{UEAC} {\liminf_{\eps\to 0}} \frac{I_{2}(ac)}{\eps}\geq 0.
\end{equation}
Recalling again \eqref{lsc}, we rewrite $I_2(s)$ as
\begin{eqnarray*}
I_2(s) &=& \int_{Q_\tau\times \uQ_\tau} \eta_{m,n}\left(\psi_{0}\left(\frac{D_{x} u_{\eps}}{|D_{x}u_{\eps}|}\right) \d |D_x^s  J_{T\varphi}(u_\eps)| - T(\underline u) \underline \z\cdot  \d D_x^s u_\eps \right)
\\
&& +\int_{Q_\tau\times \uQ_\tau} \eta_{m,n}\left(\psi_{0}\left(\frac{D_{\ux} u_{\eps}}{|D_ {\ux}u_{\eps}|}\right) \d |D_\ux ^s  J_{T\varphi}(\und u_\eps)|  -T(u)  \z\cdot  \d D_\ux ^s \und u_\eps \right)
\\ &:=& I_{2,1}(s)+I_{2,2}(s)
\end{eqnarray*}
and we only consider $I_{2,1}(s)$ ($I_{2,2}(s)$ is treated identically). Using the homogeneity of $\psi_{0}$ and Jensen's inequality, we have
$$
\begin{array}{l}
\dys - \underline \z\cdot  \int_{Q_{\tau}}\eta_{m,n}\d D_x^s u_\eps \stackrel{\eqref{H3}}\ge - \varphi(\und u) \psi_{0}\left(\int_{Q_{\tau}}\eta_{m,n}\d D_x^s u_\eps\right) \\ \\ \dys
\ge  - \varphi(\und u)  \int_{Q_{\tau}}\psi_{0}\left(\frac{D_x^s u_\eps}{|D_x^s u_\eps|}\right)\eta_{m,n}\d |D_x^s u_\eps|\,.
\end{array}
$$
Therefore, using the fact that $ \chi_{\{\und u>a\}}\leq 1$ we get
$$
\dys I_{2,1}(s) \ge \int_{Q_\tau\times \uQ_\tau} \eta_{m,n}\psi_{0}\left(\frac{D_{x} u_{\eps}}{|D_ {x} u_{\eps}|}\right)\chi_{\{\und u>a\}}\left( \d |D_x^s  J_{T\varphi}(u_\eps)| - T(\underline u) \varphi(\und u) \d|D_x^s u_\eps| \right).
$$
We split its Cantor and jump parts. Since
$$
|D_x^c  J_{T\varphi}(u_\eps)|= |J_{T\varphi}'(u_\eps)||D_x^c u_\eps| = T(u_\eps)\varphi(u_\eps)|D_x^c u_\eps|,
$$
we have
$$
I_{2,1}(c) \geq\int_{Q_\tau\times \uQ_\tau} \eta_{m,n}\psi_{0}\left(\frac{D_{x} u_{\eps}}{|D_{x} u_{\eps}|}\right)\chi_{\{\und u>a\}}\left( T(u_\eps)\varphi(u_\eps) -T(\und u) \varphi(\und u)\right) \d|D_x^c u_\eps|.
$$
 Since $T$ and $\varphi$ are (locally) Lipschitz continuous and $u,\und u$ are bounded, a positive constant $L$ exists such that
\begin{eqnarray*}
\frac{1}{\eps}{I_{2,1}(c)} & {\geq} & {-} \frac{L{K}}{\eps}  \int_{Q_\tau\times \uQ_\tau} \eta_{m,n} \chi_{\{\underline{u}{>} a\}}
  \vert u_{\eps} - \underline{u}\vert
 \d \vert D_x^c u_\eps \vert
 \\ & \ge & { - L{K} \|\eta_{m,n}\|_\infty \int_{Q_\tau\times \uQ_\tau} \chi_{\{\underline{u}{>} a\}}
 \d \vert D_x^c u_\eps \vert \stackrel{\eqref{coarea}}\to 0 \quad\mbox{as $\eps\to 0$,}}
\end{eqnarray*}
where $K$ is the maximum value of $\psi_{0}$ on $S^{N-1}$. Applying the same argument to $I_{2,2}(c)$, we conclude that
\begin{equation}\label{UEC}
{\liminf_{\eps\to 0} \frac{I_{2}(c)}{\eps} \geq 0.}
\end{equation}
Concerning the jump part, recalling that $\und u-\eps\le u_\eps\le \und u$ and reasoning as above, we have
\begin{eqnarray*}
\lefteqn{\frac{1}{\eps} I_{2,1}(j) \geq   \frac{1}{\eps} \int_{(0,\tau)\times {\uQ_\tau}}\int_{J_{u_{\eps}(t)}} \eta_{m,n}\psi_{0}\left(\frac{D_{x} u_{\eps}}{|D_{x} u_{\eps}|}\right) \chi_{\{\und u>a\}}
\times }
\\&& \qquad \times \left(\left(\int_{u_{\eps}^-}^{u_{\eps}^+}
T(s)\varphi(s)\d s\right) - T(\underline u)\varphi(\underline u)|u_\eps^+-u_\eps^-| \right)\, \d\mathcal H^{N-1}{(x)}
\\ & = &
  \frac{{1}}{\eps}\int_{(0,\tau)\times \uQ_\tau}\int_{J_{u_{\eps}(t)}} \eta_{m,n}\psi_{0}\left(\frac{D_{x} u_{\eps}}{|D _{x} u_{\eps}|}\right) \chi_{\{\und u>a\}}\times
\\ && \qquad \times
\left(\int_{u_{\eps}^-}^{u_{\eps}^+}
(T(s)\varphi(s) - T(\underline u)\varphi(\underline u))\d s\right) \d\mathcal H^{N-1}{(x)}
\\
 &\ge &  -\frac{L{K}}{\eps}\int_{(0,\tau)\times \uQ_\tau}\int_{J_{u_{\eps}(t)}} \eta_{m,n}
\chi_{\{\und u>a\}}
\left(\int_{u_{\eps}^-}^{u_{\eps}^+}
(\underline u-s)\d s\right) \d\mathcal H^{N-1}{(x)}
\\   & \ge &  -L{K}\|\eta_{m,n}\|_\infty \int_{(0,\tau)\times \uQ_\tau}\int_{J_{u_{\eps}(t)}}
\chi_{\{\und u>a\}}
\left(u_{\eps}^+ - u_\eps^-\right) \d\mathcal H^{N-1}{(x)}
\\ &=&  -L{K}\|\eta_{m,n}\|_\infty \int_{Q_\tau\times \uQ_\tau}
\chi_{\{\und u>a\}} \d |D_x^j u_\eps| \stackrel{\eqref{coarea}}\to 0 \quad\mbox{as $\eps\to 0$.}
\end{eqnarray*}
By applying the same argument on $I_{2,2}(j)$, we conclude that
\begin{equation}\label{UEJ}
\liminf_{\eps\to 0}\frac{I_{2}(j)}{\eps} \geq  0.
\end{equation}
Collecting \eqref{ineqforj}, \eqref{UEAC}, \eqref{UEC}, and \eqref{UEJ} into \eqref{2UE8} and using \eqref{ineqforj}, we conclude that
\begin{equation}\label{latter}
-\int_{Q_\tau\times \uQ_\tau}\left((\eta_{m,n})_t J_T^{\underline u}(u) + (\eta_{m,n})_\ut J_T^u(\underline u)\right)\le 0.
\end{equation}

\it (3). Conclusion.
Since \eqref{latter} does not contain any spatial gradient, we can divide it by $b-a$ and easily pass it to the limit as $m\to \infty$, $a\to 0$, and $b\to 0$, in this order. Noting that
$$
\lim_{b\to 0}\lim_{a\to 0}\frac{1}{b-a} T(s) =\sign(s)_+,
$$
$$
\lim_{b\to 0}\lim_{a\to 0}\frac{1}{b-a} J_T^l(r) \stackrel{\eqref{def-JE}} = \int_0^r \sign (s-l)_+ \d s=(r-l)_+,
$$
we obtain
\begin{equation}\label{UE8conclusion6}
-\int_{\underline{(0,\tau)}\times Q_\tau}(
(u-\underline{u})_+(\chi_n)_t + (\underline u-u)_+(\chi_n)_\ut)\leq 0,
\end{equation}
where
\begin{equation}\label{def-chin}
u=u(t,x),\quad \und u=\und u(\ut, x), \quad \chi_n = \tilde{\rho}_n(t - \ut) \phi \bigg(\frac{t+\ut}{2}\bigg), \quad (\chi_{n})_t+(\chi_n)_\ut = \tilde\rho_n \phi'.
\end{equation}
We write
\begin{eqnarray*}
-\int_{\underline{(0,\tau)}\times Q_\tau} (\underline u-u)_+\tilde\rho_n \phi'
& \stackrel{\eqref{def-chin}}= & -\int_{\underline{(0,\tau)}\times Q_\tau}
(\underline u-u)_+((\chi_n)_t+(\chi_n)_\ut)
\\ & \stackrel{\eqref{UE8conclusion6}}=& {-\int_{\underline{(0,\tau)}\times Q_\tau} \left((\underline u-u)_+ - (u-\und u)_+ \right)(\chi_{n})_t}
\\
& = & \int_{\underline{(0,\tau)}\times Q_\tau}
(u-\underline u)(\chi_n)_t \stackrel{\eqref{def-chin}}=\int_{{Q_\tau}}
u(\chi_n)_t = 0,
\end{eqnarray*}
where in the last equality we used Proposition \ref{prop-mass-cons}. Letting $n\to \infty$, we obtain
$$
\begin{array}{l}
- \displaystyle \int_{Q_\tau} (
\underline u(t,x)-{u}(t, x))_+ \,  \phi'(t) \d t \d x \leq 0.
\end{array}
$$
Since this is true for all $0 \leq \phi \in {\mathcal D}((0, \tau))$,
it implies
$$
\int_{\R^N}   (\underline u(t,x) - {u}(t, x))_+  \d x
\leq \int_{\R^N}(\underline u(0) - {u}_0)_+ \d x {\le 0}
\ \ \ \ {\rm for \ all} \ \ t \geq 0.
$$
\end{proof}

\section{The speed-limited porous medium equation}\label{S-M}\setcounter{equation}{0}

With Theorem \ref{T-sub} at hand, we can now focus on the analysis of the waiting time phenomenon. Here and in the next section we will prove Theorem \ref{Tm-1}. We begin by considering \eqref{M}, which is slightly simpler since the subsolutions we construct are continuous: they are of the form
\begin{equation}\label{F}
u(t,x)= b^{\frac{1}{1-M}}\left(\ell-\frac{1}{1+wt}\right)^{\frac{1}{1-M}} \left(1-\frac{|x|^{2}}{(1+wt)^{2}}\right)^{\frac{1}{M-1}}_{+}\,.
\end{equation}
The form of the $x$-depending factor is chosen such that the subsolution's support evolves with constant speed $w$; the exponent $1/(M-1)$ is chosen in order to ease the calculation of $\nabla u^{M-1}$, but we expect that this choice is unessential. The form of the first, $x$-independent factor is then chosen consistently, and its exponent is dictated by the homogeneity of \eqref{M} for $|\nabla u^{M-1}|\ll 1$ (see \eqref{dec1} below). The following holds:

\begin{lemma}\label{subM}
For all $b>0$, $\ell>1$, $K>0$, and $w>0$ such that
\begin{equation}\label{W}
\frac{2N(M-1)}{b}\leq w\leq \frac{1}{\sqrt{1+\frac{b^{2}}{4} (\ell -1+\frac{\ell }{K})^{2}}}\,,
\end{equation}
the function $u$ defined in \eqref{F} is a subsolution to \eqref{M} in  $\left(0,\frac{1}{wK}\right)\times \R^N$.
\end{lemma}

\begin{proof}

We need to prove that $u$ satisfies Definition \ref{def-sub} in $\left(0,\frac{1}{wK}\right)\times \R^N$. Since all other properties are obviously satisfied, we only need to check the validity of the entropy inequality \eqref{main-ineq}.

\smallskip

{\it (1) Rewriting entropy inequalities.} Since $u\in W^{1,1}((0,1/(wK))\times\R^N)$, \eqref{main-ineq} reduces to a single inequality:
\begin{eqnarray}
\nonumber
\lefteqn{S(T^0(u)) h(T^0(u),\nabla T^0(u)) + T(S^0(u)) h(S^0(u),\nabla S^0(u))}
\\ \label{L1M} &\le & - (J_{TS}(u))_t + \div(T(u)S(u) \a(u,\nabla u))
\end{eqnarray}
for all $S,T\in \mathcal T^+$. Notice that, since $\nabla S(u)=\nabla S^0(u)$ and $u=S^0(u)$ on supp$(\nabla S^0(u))$, we have
\begin{eqnarray} \nonumber
T(u) \a(u,\nabla u)\cdot\nabla S(u) &=& T(u)\, \a(u,\nabla u)\cdot\nabla S^0(u)
\\ \nonumber & = & T(S^0(u))\, \a(S^0(u),\nabla S^0(u))\cdot\nabla S^0(u)
\\ \label{new0} &=& T(S^0(u))\, h(S^0(u),\nabla S^0(u)).
\end{eqnarray}
Analogously,
\begin{equation}\label{new1}
S(u) a(u,\nabla u)\cdot \nabla T(u)= S(T^0(u)) h(T^0(u),\nabla T^0(u)).
\end{equation}
In view of  \eqref{new0} and \eqref{new1}, \eqref{L1M} translates into
\begin{eqnarray*}
\lefteqn{\frac{u \nabla u^{M-1}\cdot \nabla u}{\sqrt{1+|\nabla u^{M-1}|^2}}(S(u)T'(u) + T(u)S'(u))}
\\ &\le &  - T(u)S(u) u_t + \div\left(T(u)S(u) \frac{u \nabla u^{M-1}}{\sqrt{1+|\nabla u^{M-1}|^2}}\right),
\end{eqnarray*}
i.e.
$$
T(u)S(u)\left(u_t - \dive\left(\frac{u \nabla u^{M-1}}{\sqrt{1+|\nabla u^{M-1}|^2}}\right)\right)\le 0.
$$
Since $T(u)S(u)\ge 0$, \eqref{L1M} is equivalent to
\begin{equation}\label{L1aM}
u_t \le \div\left(\frac{u \nabla u^{M-1}}{\sqrt{1+|\nabla u^{M-1}|^2}}\right).
\end{equation}

{\it (2) Constructing subsolutions.} We look for subsolutions of the form
$$
u(t,x)=\frac{1}{A(t)}f\left(\frac{x}{B(t)}\right), \quad f(y)=(1-|y|^{2})^{\frac{1}{M-1}}, \quad  B(t)=1+wt.
$$
We notice that
\begin{equation}\label{utM}
u_{t} = -\frac{A'}{A^{2}}f - \frac{B'}{A B}{\nabla_y} f \cdot y
\end{equation}
and that, since $\nabla_y f^{M-1}=-2y$,
$$
\div\left(\frac{u \nabla u^{M-1}}{\sqrt{1+|\nabla u^{M-1}|^2}}\right)= \frac{1}{A^M B^2} \dive_y \left(\frac{f\nabla_y f^{M-1}}{D}\right) = -\frac{2}{A^M B^2} \dive_y \left(\frac{fy}{D}\right),
$$
where
$$
D:= \sqrt{1+|\nabla u^{M-1}|^2} = \sqrt{1+\frac{4|y|^2}{A^{2M-2}B^2}}.
$$
Therefore
$$
\div\left(\frac{u \nabla u^{M-1}}{\sqrt{1+|\nabla u^{M-1}|^2}}\right) = -\frac{2}{A^M B^2} \left(\frac{Nf}{D} + \frac{y\cdot \nabla_y f}{D} + |y| \frac{\d}{\d |y|}\left(\frac{f}{D}\right)\right),
$$
and after straightforward computations we obtain that
\begin{equation}\label{divM}
\div\left(\frac{u \nabla u^{M-1}}{\sqrt{1+|\nabla u^{M-1}|^2}}\right) = -\frac{2 y\cdot \nabla_y f}{A^M B^2 D} -\frac{2Nf}{A^M B^2 D^3}\left( 1+ \frac{N-1}{N}\frac{4|y|^2}{A^{2M-2}B^2}\right).
\end{equation}
Observing that $\nabla_{y} f(y)\cdot y\leq 0$, \eqref{utM} and \eqref{divM} show that \eqref{L1aM} is implied by the following two inequalities:
\begin{equation}\label{dec0}
\frac{A'}{A^{2}}\geq \frac{2N}{A^{M}B^{2}}\cdot \frac{\left(1+\frac{(N-1)}{N}\frac{4|y|^{2}}{A^{2(M-1)}B^{2}}\right)} {\left(1+\frac{4|y|^{2}}{A^{2(M-1)}B^{2}}\right)^{3/2}},
\end{equation}
\begin{equation}\label{dec2}
\frac{B'}{AB}\leq \frac{2}{ A^{M}B^{2}\left(1+\frac{4|y|^{2}}{A^{2(M-1)}B^{2}}\right)^{1/2}}\ \,.
\end{equation}
Since the second factor on the right-hand side of \eqref{dec0} is decreasing with respect to $\frac{4|y|^{2}}{A^{2(M-1)}B^{2}}$, \eqref{dec0} is in turn implied by
\begin{equation}\label{dec1}
\frac{A'}{A^{2}}\geq \frac{2N}{A^{M}B^{2}}\,, {\quad\mbox{i.e.}\quad  (A^{M-1})'B^2 \geq 2N(M-1).}
\end{equation}
Therefore, choosing
$$
A(t)=b^{\frac{1}{M-1}}\left(\ell-\frac{1}{1+wt}\right)^{\frac{1}{M-1}},
$$
we see that \eqref{dec1} is satisfied if
\begin{equation}\label{cM1}
(A^{M-1})'B^{2}=bw\geq 2N(M-1)\,.
\end{equation}
On the other hand, \eqref{dec2} is implied by
\begin{equation}\label{cM2}
w=B'\leq \min_{t,y}\frac{1}{\sqrt{\frac{A^{2(M-1)}B^{2}}{4}+ |y|^{2}}}= \frac{1}{\sqrt{1+\frac{b^{2}}{4}\left(\ell-1 +\frac{\ell}{K}\right)^{2}}}\,,
\end{equation}
where in the last step we used
$$
A^{2(M-1)}B^{2}=b^{2}(\ell-1+\ell wt)^{2}\leq b^{2}\left(\ell-1+\frac{\ell}{K}\right)^{2} \quad \mbox{for all $t\in (0,\frac{1}{wK})$.}
$$
Combining the conditions in \eqref{cM1} and \eqref{cM2}, we obtain the condition in  \eqref{W} and the proof is complete.
\end{proof}

By scaling, we obtain the following family of subsolutions.
\begin{corollary}\label{scal}
If $b>0$, $\ell> 1$ $K>0$, and $w>0$  are such that \eqref{W} holds, then for any $s>0$ and any $\xi\in\rn$ the function
\begin{equation}\label{Fs}
\und u(t,x)= b^{\frac{1}{1-M}}\left(\frac{\ell}{s}-\frac{1}{s+wt}\right)^{\frac{1}{1-M}} \left(1-\frac{|\xi-x|^{2}}{(s+wt)^{2}}\right)^{\frac{1}{M-1}}_{+}\,,
\end{equation}
is a subsolution to \eqref{M} in $(0,\frac{s}{wK})\times \R^N$.
\end{corollary}
\begin{proof}
We use the scaling invariance of \eqref{M} with respect to the following transformations:
$$
u=U\und u, \ \ \ x=U^{M-1}\und x,\ \ \ \text{and}\ \ \ t=U^{M-1}\und t\,.
$$
By Lemma \ref{subM}, provided \eqref{W} holds,
\begin{eqnarray*}
\und u (\und t , \und x)  &= &\frac{1}{U} u (U^{M-1}\und t, U^{M-1}\und x)
\\
& = &
{
\frac{b^{\frac{1}{1-M}}}{U}\left(\ell-\frac{1}{1+w U^{M-1}\und t}\right)^{\frac{1}{1-M}} \left(1-\frac{U^{2M-2}|\und x|^{2}}{(1+wU^{M-1}\und t)^{2}}\right)^{\frac{1}{M-1}}_{+}
}
\\
& = &
b^{\frac{1}{1-M}}\left(\frac{\ell}{U^{1-M}}-\frac{1}{U^{1-M}+w \und t}\right)^{\frac{1}{1-M}} \left(1-\frac{|\und x|^{2}}{(U^{1-M}+w\und t)^{2}}\right)^{\frac{1}{M-1}}_{+}
\end{eqnarray*}
is a subsolution to \eqref{M} in $\left(0, \frac{U^{1-M}}{wK}\right)\times \R^N$. The result follows replacing $U^{1-M}$ by $s$, $(\und t,\und x)$ by $(t,x)$, and using  translation invariance in space.
\end{proof}

We are ready to prove Theorem  \ref{Tm-1} in the case of Equation \eqref{M}.

\begin{proof}[Proof of Theorem \ref{Tm-1}: Equation \eqref{M}.]
Up to a translation and a rotation, we assume without losing generality that $x_0=0$ and that $v_{0}=(-1, 0, \dots ,  0)$. We consider the case  $L<+\infty$ in \eqref{crit-M} (from which the case $L=+\infty$ follows immediately).

\smallskip

We wish to choose  the parameters in Lemma \ref{subM} so that the function $\und u$ given in \eqref{Fs} is a subsolution to \eqref{M} and $\und u(0) \le u_{0}$. We fix
\begin{equation}\label{def-K}
K=2N(M-1).
\end{equation}
In Corollary \eqref{scal} we let $w=K/b$, so that condition \eqref{W} reduces to
$$
\frac{K}{b}\leq \frac{1}{\sqrt{1+\frac{b^{2}}{4} (\ell -1+\frac{\ell }{K})^{2}}}\,,
$$
that is,
\begin{equation}\label{WM}
\frac{4}{b^2} + \left(\ell -1+\frac{\ell }{K}\right)^{2}\leq \frac{4}{K^2}.
\end{equation}
By \eqref{crit-M}, $R>0$ exists such that
$$
u(0,x)\geq \frac{L}{2}|x|^{\frac{2}{M-1}}\, \ \ \ \ \text{in}\ \ B(Rv_{0}, R)\subseteq {\rm supp} (u_{0})\,.
$$
Let $\xi=r_1 v_0$ with
$$
0<r_{1}\leq \min\{R,L^{1-M}\}.
$$
Then $\und u(0)\le u_0$ if
$$
b^{\frac{1}{1-M}}\left(\frac{\ell-1}{s}\right)^{\frac{1}{1-M}} \left(1-\frac{|r_{1}v_{0}-x|^{2}}{s^{2}}\right)^{\frac{1}{M-1}}_{+}\leq  \frac{L}{2}|x|^{\frac{2}{M-1}}\,\ \ \text{in $B(Rv_{0}, R)$}\,,
$$
which may be rewritten as
\begin{equation}\label{desi}
\frac{1}{sb(\ell-1)}\left(s^{2}- |r_{1}v_{0}-x|^{2}\right)_{+}\leq  \frac{L^{M-1}}{2^{M-1}}|x|^2\,\ \ \text{in $B(Rv_{0}, R)$}\,.
\end{equation}
We are now going to choose  $\ell> 1$, $s>0$, and $b>0$ such that \eqref{WM} and \eqref{desi} hold. Let
\begin{equation}\label{condb}
b=\frac{\alpha 2^{M-1} L^{1-M}}{s(\ell-1)}>0,
\end{equation}
where $\alpha>0$, {\em depending only on $N$ and $M$}, will be chosen below. Then \eqref{desi} reduces to
$$
s^{2}\leq |r_{1}v_{0}-x|^{2}+\alpha|x|^{2}\quad \mbox{in $B(Rv_{0}, R)$.}
$$
Since the minimum value of the right-hand side is attained at $x=r_1 v_0/(\alpha+1)$, \eqref{desi} is in turn implied by
\begin{equation}\label{conds}
s:={\frac{\alpha}{\alpha+1} r_1.}
\end{equation}
In view of \eqref{condb} and \eqref{conds}, \eqref{WM} may be rewritten as
\begin{equation}\label{WM1}
\frac{4 r_1^2 (\ell-1)^2}{(\alpha+1)^2 4^{M-1} L^{2(1-M)}} + \left(\ell -1+\frac{\ell }{K}\right)^{2}\leq \frac{4}{K^2}.
\end{equation}
Since $r_1\le L^{1-M}$, \eqref{WM1} is implied by
\begin{equation*}
\frac{4(\ell-1)^2}{(\alpha+1)^2 4^{M-1}} + \left(\ell -1+\frac{\ell }{K}\right)^{2}\leq \frac{4}{K^2}.
\end{equation*}
Therefore we can choose $\ell$, {\em depending only on $N$ and $M$}, so close to $1$ that \eqref{WM} holds. Hence $\und u$ given in \eqref{Fs} is a subsolution to \eqref{M} in  $\left(0,\frac{s}{wK}\right)\times\rn$ and $\und u(0)\le u_0$.

\smallskip

We finally estimate $t_{\ast}$. The time $T_{u}$ at which the support of $\und u$ reaches $x_{0}=0$ is given by
$$
T_{u}\stackrel{\eqref{Fs}}= \frac{r_{1}-s}{w} \stackrel{\eqref{conds}}= \frac{r_1}{(\alpha+1)w}=\frac{s}{\alpha w}.
$$
{We now choose $\alpha=2K$ (recall \eqref{def-K}), so that} $T_{u}<\frac{s}{wK}$. Therefore $\und u$ does reach $x_0=0$, and recalling that $w=K/b$ we obtain
$$
T_{u}=\frac{s}{\alpha w}= \frac{s b}{\alpha K} \stackrel{\eqref{condb}}= \frac{2^{M-1} L^{1-M}}{K(\ell-1)} =:WL^{1-M}\,,
$$
with $W$ depending only on $N$ and $M$. Therefore $t_*\le T_*$, which concludes the proof of Theorem \ref{Tm-1} in the case of Equation \eqref{M}.
\end{proof}

\section{The relativistic porous medium equation}\setcounter{equation}{0}

As we observed in the introduction, in case of \eqref{m} it is natural to look for subsolutions with a jump discontinuity at the boundary of their support:
\begin{equation}\label{def-m-scaled-ok}
u(t,x)=\frac{1}{A(t)}\left(1+\sqrt{r(t)^2-|x|^2}\right)\chi_{Q_0}(t,x),
\end{equation}
where
\begin{eqnarray}\label{at}
A(t) &=& [(m-1)(1+\gamma t)]^{\frac{1}{m-1}}\,,
\\
\label{rt}
r(t) &=& r_{0}+\frac{1}{\gamma (m-1)}\log(1+\gamma t)\,,
\\ \label{defQ}
Q_{0} &=& \left\{(t,x): \ t\in (0,T),\ x\in B(0,r(t)) \right\}.
\end{eqnarray}
The square root in \eqref{def-m-scaled-ok} is chosen for convenience and we expect that it can be replaced by any exponent smaller than $1$ (see \eqref{lkj1}). As we discussed in the introduction, the functions $A$ and $r$ are chosen so that $r'=A^{1-m}$ --which is dictated by a Rankine-Hugoniot condition at the jump set $\partial Q_0$, see \eqref{crucial1} below-- and that $(A^{m-1})'$ is constant --which is dictated by homogeneity, see \eqref{crucial2} below.

\smallskip

In this section we prove:

\begin{proposition}\label{bbb-pro}
Let $N\ge 1$, $m>1$,  $T>0$ and $r_{1}>0$. Then there exist a value $\gamma_{0}\geq 1$ such that the function $u$ defined by \eqref{def-m-scaled-ok}-\eqref{defQ} is a subsolution to \eqref{m}  for any $\gamma\geq \gamma_{0}$ and any {$r_0\in \left[\frac{r_{1}}{2}, r_{1}\right]$}.
\end{proposition}

\begin{proof}

{\it (1) Splitting entropy inequalities.} Since $u$ has the form given in  \eqref{def-m-scaled-ok}, both $[(J_{TS}(u))_t](t,\cdot)$ and $[\div(T(u)S(u) \a(u,\nabla u))](t,\cdot)$ are signed measures for all $t\ge 0$. Therefore, \eqref{main-ineq} may be rewritten in form of two inequalities between measures, splitting Lebesgue and singular parts:
\begin{eqnarray}
\nonumber
\lefteqn{S(T^0(u)) h(T^0(u),\nabla T^0(u)) + T(S^0(u)) h(S^0(u),\nabla S^0(u))}
\\ \label{L1} &\le & - (J_{TS}(u))_t + (\div(T(u)S(u) \a(u,\nabla u)))^{a}
\end{eqnarray}
and
\begin{eqnarray}
\nonumber
\lefteqn{|D_x^s J_{S\varphi}(T^0(u))| + |D_x^s J_{T\varphi}(S^0(u))|}
\\ \label{S1}&\le& - (D_t(J_{TS}(u)))^s + (\div (T(u)S(u) \a(u,\nabla u)))^{s},
\end{eqnarray}
where  $S,T\in \mathcal T^+$ and $\mu^a$, resp. $\mu^s$, denote the absolutely continuous, resp. singular, part of the Radon-Nikodym decomposition of a measure $\mu$ with respect to the Lebesgue measure (see \cite[Theorem 1.28]{AFP}). We discuss the two inequalities separately.

\smallskip

{\it (2) Subsolutions on the jump set.} Let us check \eqref{S1}. We note that the singular parts are concentrated on $|x|=r(t)$, where $u$ has jumps with
$u^+(t,x)= 1/A(t)$ and $u^-(t,x) = 0$ (cf. \S\ref{ss-not}).
Hence
\begin{eqnarray} \nonumber
|D_x^s J_{S\varphi}(T^0(u))| &=&  \left(\int_{T^0(u^-)}^{T^0(u^+)}S(\sigma)\sigma^m\d \sigma \right) \mathcal H^{N-1} \llcorner {\{|x|=r\}}
\\ \nonumber
&=& \left(\int_{u^-}^{u^+}S(\sigma)T'(\sigma)\sigma^m\d \sigma \right) \mathcal H^{N-1} \llcorner {\{|x|=r\}}
\end{eqnarray}
and, analogously,
\begin{eqnarray}\nonumber
|D_x^s J_{T\varphi}(S^0(u))| &=& \left(\int_{u^-}^{u^+}S'(\sigma)T(\sigma)\sigma^m\d \sigma \right) \mathcal H^{N-1} \llcorner {\{|x|=r\}}.
\end{eqnarray}
Therefore, since $u^-=0$,
\begin{eqnarray}\label{jk2}
|D_x^s J_{S\varphi}(T^0(u))| + |D_x^s J_{T\varphi}(S^0(u))|
= \left(\int_{0}^{u^+}(S(\sigma)T(\sigma))'\sigma^m\d \sigma \right) \mathcal H^{N-1} \llcorner {\{|x|=r\}}\,.
\end{eqnarray}
For the first term on the right-hand side of \eqref{S1}, arguing as in \cite[proof of (3.3)]{G15}  we obtain
\begin{eqnarray}
\nonumber
\lefteqn{- (D_t(J_{TS}(u)))^s = - {r'}\left(\int_{0}^{u^+} S(\sigma)T(\sigma)\d \sigma\right) \mathcal H^{N-1}\llcorner {\{|x|=r\}} }
\\ \label{jk3} &=& -r' \left(u^+ T(u^+)S(u^+)\  {-} \int_{0}^{u^+} (S(\sigma) T(\sigma))' \sigma\d \sigma\right)\mathcal H^{N-1}\llcorner {\{|x|=r\}},
\end{eqnarray}
where we used one integration by parts in the last equality. Finally, for the the second term on the right-hand side of \eqref{S1}, we have
\begin{eqnarray*}
(\dive (T(u)S(u) \a(u,\nabla u)))^s = -  \lim_{|x|\to r(t)^-} T(u)S(u) \a(u,\nabla u) \cdot \frac{x}{r} \mathcal H^{N-1}\llcorner {\{|x|=r\}}.
\end{eqnarray*}
The fact that $\nabla u$ blows up at the boundary implies that
$$
\lim_{|x|\to r(t)^-} \a(u,\nabla u)\cdot\frac{x}{r}= -(u^+)^{m}.
$$
Therefore
\begin{equation}
\begin{array}{l}\label{jk4}
{(\dive (T(u)S(u) \a(u,\nabla u)))^s} =
\dys T\left(u^+\right)S\left(u^+\right) (u^+)^{m} \ \mathcal H^{N-1}\llcorner {\{|x|=r\}}.
\end{array}
\end{equation}
Combining \eqref{jk2}, \eqref{jk3}, and \eqref{jk4}, we see that \eqref{S1} is equivalent to
\begin{equation}\label{lk}
\int_{0}^{u^+}(S(\sigma)T(\sigma))'\sigma(\sigma^{m-1} - r') \d \sigma
\le  u^+ T(u^+)S(u^+)( (u^+)^{m-1}-r').
\end{equation}
Since $r$ and $A$ have been chosen such that \begin{equation}
\label{crucial1}
r'= A^{1-m} =(u^+)^{m-1},
\end{equation}
the left-hand side of \eqref{lk} is negative (since $(ST)'\ge 0$) and the right-hand side of \eqref{lk} is {zero}. Hence \eqref{S1} holds.

\smallskip

{\it (3) Subsolution in the bulk.} In $Q_0$, arguing as in Step (1) in the proof of Lemma \ref{subM}, we obtain that \eqref{L1} is equivalent to
\begin{equation}\label{L1a}
u_t \le \dive\left(\frac{u^m\nabla u}{\sqrt{u^2+|\nabla u|^2}}\right)\quad\mbox{in $Q_0$.}
\end{equation}
We look for subsolutions of the form \eqref{def-m-scaled-ok}. For notational convenience, we let
\begin{eqnarray*}
\eta(t,x) &:=& \sqrt{r(t)^2-|x|^2},
\\
D(t,x) &:=& \sqrt{\eta^2(1+\eta)^2+|x|^2},
\\
E(t,x) &:=& (1+\eta)^2+\eta(1+\eta)-1.
\end{eqnarray*}
Then, we compute
$$
u_{t}=-\frac{A'}{A^{2}}(1+\eta)+\frac{rr'}{A\eta}
$$
and, since $\nabla \eta=-x/\eta$,
\begin{eqnarray*}
\lefteqn{{A(t)^m} \dive\left(\frac{u^m\nabla u}{\sqrt{u^2+|\nabla u|^2}}\right) = - \dive\left(\frac{(1+\eta)^m x}{D}\right)}
\\
&=&  - \left( N \frac{(1+\eta)^m} {D} + |x|\left(\frac{d}{d|x|}\frac{(1+\eta)^m} {D}\right)\right)
\\ &=&
- \left( N \frac{(1+\eta)^m} {D} + \frac{|x|^2 (1+\eta)^{m}E}{D^3} - \frac{m|x|^2(1+\eta)^{m-1}}{\eta D}
\right)
\\&=:& F(|x|,r).
\end{eqnarray*}
Therefore {$u$} satisfies \eqref{L1a} if and only  if
\begin{equation}\label{subif}
-A^{m-2}A'(1+\eta)+A^{m-1}\frac{r}{\eta}r' {\stackrel{\eqref{crucial1}}= -A^{m-2}A'(1+\eta)+ \frac{r}{\eta}} \leq F(r,|x|)\,.
\end{equation}
Since $A$ has been chosen such that
\begin{equation}\label{crucial2}
A^{m-2}A'=\gamma,
\end{equation}
it follows from \eqref{subif} that \eqref{L1a} is satisfied if and only if
\begin{equation}\label{gmax}
\gamma\geq  \frac{\frac{r}{\eta}-F(r,|x|)}{1+\eta}:=G(r,|x|)\,.
\end{equation}
Observe that $\gamma \mapsto \frac{1}{\gamma}\log(1+\gamma t)$ is nonincreasing. Therefore, for $\gamma\geq 1$ and $t\leq T$ we have
\begin{eqnarray*}
\frac{r_{1}}{2} \leq r(t) &=& r_{0}+\frac{1}{\gamma (m-1)}\log(1+\gamma t) \leq r_{0}+\frac{1}{(m-1)}\log(1+ t)
\\ &\leq& r_{1}+\frac{1}{ (m-1)}\log(1+ T)
\end{eqnarray*}
Hence, by \eqref{gmax}, \eqref{L1a} is satisfied if
$$
\gamma\geq \gamma_0:=\sup_{(r,|x|)\in H} G(r,|x|)\,,
$$
where
$$
H=\left\{(r,y)\in \R^{2}_+: \ r\in\left[\frac{r_{1}}{2},r_{1}+\frac{1}{(m-1)}\log(1+T)\right],\  y\in [0,r)\right\}.
$$
Since $\eta=\sqrt{\rho^{2}-y^{2}}\to 0$ as $(\rho,y)\to (r,r)$, we have
$$
G(\rho,y) \eta  =\frac{\rho - my^{2}(1+\eta)^{m-1}D^{-1}}{1+\eta}+ o(1)\to  r(1-m)<0 \quad \mbox{as $(\rho,y)\to (r,r)$.}
$$
In addition, $G(r,y)$ is continuous in $H$: therefore $\gamma_0$ is finite. Since $\gamma_0$ only depends on $N$, $m$, $T$, and $r_1$, the proof is complete.
\end{proof}

Using the invariance of \eqref{m} with respect to
$$
\hat u=U u, \ \ \ \text{and}\ \ \ \hat t=U^{m-1} t\,
$$
and the translation invariance of \eqref{m} with respect to $x$, we immediately obtain:

\begin{corollary}\label{bbb-cor}
Let $N\ge 1$, $m>1$,  $T>0$, and $r_{1}>0$. Then $\gamma_0 {\ge 1}$ exists such that
\begin{equation}\label{def-m}
\underline u(t,x)= U u(U^{m-1}t, x-\xi)
\end{equation}
is a subsolution to \eqref{m} in $(0,U^{1-m}T)\times\rn$ for any $\gamma\geq \gamma_{0}$, any $r_0\in [r_{1}/2, r_{1}]$, {any $U>0$, and any $\xi\in \rn$}, where $u$ is defined by \eqref{def-m-scaled-ok}{-\eqref{defQ}}. \label{scalm}
\end{corollary}

Unlike \eqref{M}, \eqref{m} has no scaling invariance with respect to the spatial variables $x$. For this reason, the lower bound $\gamma_0$ on $\gamma$ depends on $r_1$ (and $T$). Nevertheless, as the following proof shows, a careful choice of the parameters $r_0$, $T$, $\gamma$, and $U$ permits to obtain a waiting time bound which is independent of spatial legnthscales or parameters, such as $r_1$.

\begin{proof}[Proof of Theorem \ref{Tm-1}: Equation \eqref{m}.]
As for \eqref{M}, we may assume that $x_0=0$, $v_{0}=(-1, 0, \dots ,  0)$, and $L<+\infty$ in \eqref{crit-m}.

\smallskip

Let $\xi=r_1 v_0$ and $T={4^{m+1}L^{1-m}}$. We will choose the parameters $r_1>0$, {$r_0\in[r_1/2,r_1]$,} and $U>0$ in Corollary \ref{scalm} such that the function $\underline u(t,x)$ in \eqref{def-m} is a subsolution to \eqref{m} up to time $U^{1-m}T$. By \eqref{crit-m}, $R>0$ exists such that
$$
u(0,x)\geq \frac{L}{2}|x|^{\frac{1}{m-1}}\, \quad \text{in}\ \ B(Rv_{0}, R).
$$
Hence, provided that $r_1\le R$, it suffices to verify that
\begin{equation}\label{suffi}
\underline u(0,x)^{m-1} = \frac{U^{m-1}}{m-1}\left(1+\sqrt{r_{0}^{2}-|x-r_{1}v_{0}|^{2}}\right)^{m-1} \leq \frac{L^{m-1}}{2^{m-1}}|x| \quad {\mbox{in $B(r_1v_0,r_0)$}.}
\end{equation}
Since $|x|\ge r_1-r_0$ in $B(r_1v_0,r_0)$, \eqref{suffi} is implied by
\begin{equation}\label{suffi2}
\frac{U^{m-1}}{m-1}(1+r_{0})^{m-1}\leq\frac{L^{m-1}}{2^{m-1}}(r_{1}-r_{0}).
\end{equation}
We fix
$$
r_{1}\le \min\left\{{R}, \frac{1}{m-1},1\right\},\quad \gamma = \max (\gamma_{0},2), \quad\mbox{and}\quad r_{0}=r_{1}\frac{\gamma-1}{\gamma}
$$
(note that $\frac{r_{1}}{2}\leq r_{0}<r_{1}$ since $\gamma\ge 2$). Then  \eqref{suffi2} reduces to
$$
\frac{U^{m-1}}{m-1}\left(1+ r_{1}\frac{\gamma-1}{\gamma}\right)^{m-1}\leq\frac{L^{m-1}}{2^{m-1}}\ \frac{r_{1}}{\gamma}\,,
$$
which, since $r_1\le 1$, is implied by
$$
\frac{{2^{m-1}}U^{m-1}}{m-1}\leq\frac{L^{m-1}}{2^{m-1}}\ \frac{r_{1}}{\gamma},
$$
which in turn holds true choosing
\begin{equation}\label{iu}
U^{m-1}=(m-1)\frac{L^{m-1}}{{4}^{m-1}}\ \frac{r_{1}}{\gamma}.
\end{equation}
The time $T_{u}$ at which the support of $\underline u$ reaches $x_0=0$ is implicitly defined by
$$
r_{1}={r_{1}}\frac{\gamma-1}{\gamma}+\frac{1}{\gamma (m-1)}\log (1+\gamma U^{m-1}T_{u}),
$$
or, equivalently,
$$
r_{1}=\frac{1}{(m-1)}\log (1+\gamma U^{m-1}T_{u}),
$$
that is, recalling that $(m-1)r_{1}\leq1$,
$$
T_{u}= \frac{e^{ (m-1)r_{1}}-1}{\gamma U^{m-1}} \stackrel{\eqref{iu}}= {4}^{m-1}L^{1-m}\frac{e^{ (m-1)r_{1}}-1}{(m-1)r_{1} }\leq {4}^{m} L^{1-m} {:=W L^{1-m}<T}.
$$
Therefore $t_*\le W L^{1-m}$, which concludes the proof of Theorem \ref{Tm-1}.
\end{proof}

\ignore{
\section*{Appendix}
\setcounter{equation}{0}
\renewcommand{\theequation}{A.\arabic{equation}}
\renewcommand{\theassumption}{A.\arabic{assumption}}
\renewcommand{\thetheorem}{A.\arabic{theorem}}
}

\bigskip
\footnotesize
\noindent\textit{Acknowledgments.} {The second and third author acknowledge partial support by the Spanish MEC and FEDER project MTM2015-70227-P. The third author has been partially supported by the Gruppo Nazionale per l'Analisi Matematica, la Probabilit\`a e le loro Applicazioni (GNAMPA) of the Istituto Nazionale di Alta Matematica (INdAM).}


\begin{thebibliography}{SK}



%


\bibitem{AFP}{Ambrosio L., Fusco N., Pallara D. (2000). {\it Functions of bounded variation and free discontinuity problems.} Oxford Mathematical Monographs.}




 \bibitem{NA-ACM} Andreu, F., Caselles, V., Maz\'on, J. M. (2005) A strongly degenerate quasilinear elliptic equation. {\it  Nonlinear Anal.} 61, no. 4, 637--669.

\bibitem{ACM}{Andreu F., Caselles V., Maz\'on J. M.  (2005). The Cauchy problem for a strongly degenerate quasilinear equation. {\it J. Euro. Math. Soc.} {7}, no. 3, 361--393.}


\bibitem{ACM-JDE} Andreu, F., Caselles, V., Maz\'on, J. M. (2010). A Fisher-Kolmogorov equation with finite speed of propagation. {\it J. Differential Equations} 248 , no. 10, 2528--2561.


\bibitem{ACMM-ARMA}{Andreu F., Caselles V., Maz\'on J. M.,  Moll S. (2006).  Finite propagation speed for limited flux diffusion equations. {\it Arch. Rat. Mech. Anal.} {182},  269--297. }
\bibitem{ACMM-JEE}Andreu F., Caselles V., Maz\'on J. M.,  Moll S. (2007). A diffusion equation in transparent media. {\it J. Evol. Equ.} 7, no. 1, 113--143

\bibitem{ACMM-MA}Andreu F., Caselles V., Maz\'on J. M.,  Moll S. (2010). The Dirichlet problem associated to the relativistic heat equation. {\it Math. Ann.} 347, no. 1, 135--199.

\bibitem{acms1} Andreu F., Calvo J.,  Maz\'on J. M. , Soler J. (2012). On a nonlinear flux-limited equation arising in the transport of morphogens. {\it J. Differential Equations} {252}, no. 10, 5763--5813.

\bibitem{ACMSV}{Andreu F., Caselles V., Maz\'on J. M., Soler J., Verbeni M. (2012). Radially symmetric solutions of a tempered diffusion equation. A porous media, flux-limited case. {\it SIAM J. Math. Anal.} {44}, 1019--1049.}


\bibitem{BW} Bellomo N., Winkler M. A degenerate chemotaxis system with flux limitation: Maximally extended solutions and absence of gradient blow-up.
    {\it Comm. Partial Differential Equations}, in press.

\bibitem{bdp}  Bertsch M.,  Dal Passo R.  (1992). Hyperbolic phenomena in a strongly degenerate parabolic equation. {\it Arch. Ration. Mech. Anal.} {117}, 349--387.



\bibitem{bl2}{Blanc P.  (1993). On the regularity of the solutions of some degenerate parabolic equations. {\it Comm. Partial Differential Equations} {18}, 821--846.}

\bibitem{br} Brenier Y. (2003). {\it Extended Monge-Kantorovich theory.} In: Optimal Transportation and Applications, Lectures given at the C.I.M.E. Summer School help in Martina Franca, L.A. Caffarelli and S. Salsa (eds.), Lecture Notes in Math. {1813}, Springer-Verlag, 91--122.

\bibitem{c15}{Calvo J. (2015). Analysis of a class of degenerate parabolic equations with saturation mechanisms. {\it SIAM J. Math. Anal.} {47}, no. 4, 2917--2951. }


\bibitem{cccss2} Calvo J.,  Campos J., Caselles V.,  S\'anchez O.,  Soler J. (2015). Flux-saturated porous media equations and applications. {\it EMS Surv. Math. Sci.} {2}  no. 1, 131--218.

\bibitem{cccss} {Calvo J.,  Campos J., Caselles V.,  S\'anchez O.,  Soler J.  Pattern formation in a flux limited reaction--diffusion equation of porous media type.  {\it Invent. Math.} 206 (2016), 57-108. }

\bibitem{CC} Calvo, J., Caselles, V. (2013) Local-in-time regularity results for some flux-limited diffusion equations of porous media type. {\it Nonlinear Anal.} 93, 236--272.

\bibitem{cmsv} Calvo J., Maz\'on J. M.,  Soler J., Verbeni M.  (2011). Qualitative properties of the solutions of a nonlinear flux-limited equation arising in the transport of morphogens. {\it Math. Models Methods Appl. Sci.} {21}, suppl. 1, 893--937.

\bibitem{CCM}{Carrillo J. A., Caselles V.,  Moll S. (2013). On the relativistic heat equation in one space dimension. {\it Proc. London Math. Soc.} {107}, 1395--1423.}

\bibitem{C-DCDS}{Caselles V. (2011). An existence and uniqueness result for flux limited diffusion equations. {\it Discrete Contin. Dyn. Syst.} {31}, no. 4, 1151--1195. }

\bibitem{C-JDE}{Caselles V.  (2011). On the entropy conditions for some flux limited diffusion equations. {\it J. Differential Equations} {250}, no. 8, 3311--3348.}

\bibitem{C-PM}{Caselles V. (2013). Flux limited generalized porous media diffusion equations. {\it Publicacions Matem\`{a}tiques} {57}, 155--217. }

\bibitem{Caselles-Annali}{Caselles V.  (2015). Convergence of flux limited porous media diffusion equations to its classical counterpart. {\it Annali della Scuola Normale Superiore di Pisa (5)} {14}, no. 2, 481--505.}

\bibitem{ckr} Chertock A.,  Kurganov A.,  Rosenau P. (2003).  Formation of discontinuities in flux-saturated degenerate parabolic equations. {\it Nonlinearity} {16},  1875--1898.

\bibitem{dp}{Dal Passo, R. (1993). Uniqueness of the entropy solution of a strongly degenerate parabolic equation.
{\it Comm. Partial Differential Equations} 18, 265--279.}

\bibitem{DGG}{Dal Passo R., Giacomelli L., Gr\"un G. (2001). A waiting time phenomenon for thin film equations. {\it Ann. Scuola Norm. Sup. Pisa Cl. Sci. (4)} {30} no. 2, 437--463. }


\bibitem{DuM}{Duderstadt J. J.,  Moses G. A.  (1982). {\it Inertial confinement fusion.}  John Wiley \& Sons. New York.}



\bibitem{F}{ Fisher J. (2014). Upper bounds on waiting times for the thin-film equation: the case of weak slippage. {\it Arch. Ration. Mech. Anal.} {211}, 771--818. }

\bibitem{G15}{Giacomelli L. (2015). Finite speed of propagation and waiting time phenomena for degenerate parabolic equations with linear-growth Lagrangian. {\it SIAM J. Math. Anal.} {47}, 2426--2441. }

\bibitem{GG}{Giacomelli L., Gr\"un G. (2006). Lower bounds on waiting times for degenerate parabolic equations and systems. {\it Interfaces Free Bound.} {8}, no. 1, 111--129. }

\bibitem{G}{Gr\"un G. (2004). Droplet spreading under weak slippage: the waiting time phenomenon. {\it  Ann. Inst. H. Poincar\'{e} Anal. Non Lin\'{e}aire } {21}, no. 2, 255--269. }


\bibitem{r1}  Rosenau P. (1990). Free energy functionals at the high gradient limit. {\it Phys. Rev. A} {41}, 2227--2230.
\bibitem{r2} Rosenau P. (1992). Tempered diffusion: a transport process with propagating front and inertial delay. {\it Phys. Rev. A} {46}, 7371--7374.


\bibitem{V} V\'azquez J. L.  (2007). {\it The porous medium equation. Mathematical theory.} Oxford Mathematical Monographs. The Clarendon Press, Oxford University Press, Oxford.
\end{thebibliography}
\end{document}